\documentclass[a4paper,12pt]{amsart}

\usepackage{amsmath, amssymb, amsthm, amsfonts}
\usepackage{mathtools}
\usepackage{enumitem}
\usepackage[a4paper, hmargin=2.8cm, vmargin={3.5cm, 3.2cm}]{geometry}
\usepackage{hyperref}
\usepackage[initials, msc-links]{amsrefs}
\usepackage{tikz}

\usetikzlibrary{angles,quotes}
\usepackage{float}
\usetikzlibrary{arrows.meta, decorations.pathmorphing,calc, patterns.meta}

\hypersetup{
    colorlinks,
    linkcolor={blue},
    citecolor={blue},
    urlcolor={black},
}

\usepackage{graphicx}

\usepackage{subcaption}
\usetikzlibrary{decorations.pathmorphing,calc}
\setlist{topsep=6pt, itemsep=0.4em, after=\vspace{2pt}}
\newtheorem{thm}{Theorem}
\newtheorem*{thm*}{Theorem}

\newtheorem{lem}{Lemma}
\newtheorem*{lem*}{Lemma}

\theoremstyle{definition}

\newtheorem{rem}{Remark}

\newcommand{\bC}{\mathbb{C}}

\newcommand{\bN}{\mathbb{N}}

\newcommand{\calW}{\mathcal{W}}

\newcommand{\capacity}{\operatorname{cap}}

\renewcommand{\Re}{\operatorname{Re}}
\renewcommand{\Im}{\operatorname{Im}}

\numberwithin{equation}{section}
\author{Erwin Mi\~{n}a-D\'{i}az \and Olof Rubin \and Aron Wennman}
\begin{document}

\title[Norms of Chebyshev and Faber polynomials]
{Norms of Chebyshev and Faber polynomials on curves with corners and cusps}

\begin{abstract} We  prove that the $n$th Chebyshev polynomial $T_{n}$ of a piecewise Dini-smooth Jordan curve $\Gamma$  satisfies 
\[
\lim_{n\to\infty}\frac{\|T_{n}\|_{\Gamma}}{\mathrm{cap}(\Gamma)^n}=1,
\]
where $\|\cdot\|_\Gamma$ is the supremum norm over $\Gamma$ and $\mathrm{cap}(\Gamma)$ its logarithmic capacity.
This extends earlier results for smooth curves to curves with corner singularities, including cusps.
The proof makes use of weighted Faber polynomials, which we analyze using 
a Fourier analytic representation of the standard Faber polynomials due to Pommerenke. 
We moreover obtain new asymptotic bounds for the norm of Faber polynomials 
which are sharp if, for instance, all corners
have exterior angle greater than  \(\pi\).
\end{abstract}
\maketitle

\section{Introduction}
\noindent Given a compact set \(K \subset \bC\) we let \(\|\cdot\|_K\) 
denote the corresponding supremum norm. If \(K\) contains infinitely many points,  there exists, for every \(n\in \bN\), 
a unique monic polynomial of degree \(n\) minimizing \(\|\cdot\|_K\)
called the \(n\)th \textit{Chebyshev polynomial} of \(K\), 
and we denote it by \(T_{n}\). The literature on this topic is extensive and we refer the reader to 
\cites{achieser56,andrievskii17,faber20, smirnov-lebedev68, christiansen-simon-zinchenko-I, christiansen-simon-zinchenko-III, 
christiansen-simon-zinchenko-IV, eichinger17, novello-schiefermayr-zinchenko21, 
totik-varga15, totik-yuditskii15, totik14, widom69} 
and the references therein for a background. A central question is to understand 
the asymptotic behavior of \(\|T_n\|_K\) as \( n \) tends to infinity and it turns 
out that this limiting behavior is intimately connected to the logarithmic 
capacity of the set \( K \), denoted by \( \capacity(K) \). Szeg\H{o}'s inequality states that if $K$ has positive logarithmic capacity, then
\begin{equation}
\calW_{n}(K) \coloneqq \frac{\|T_n\|_K}{\capacity(K)^n}\geq 1,
\label{eq:szego_inequality}
\end{equation}
where \(\calW_n\) is the so-called {\em Widom factor}. A complementary result, 
established in various forms by Faber \cite{faber20}, 
Fekete \cite{fekete23}, and Szeg\H{o} \cite{szego24}, asserts that
\begin{equation}
\lim_{n\rightarrow \infty}\calW_n(K)^{1/n} =1. 
\label{eq:ffs_theorem}
\end{equation}

Without prior restrictions, the Widom factors can exhibit any kind of 
sub-exponential growth \cite{goncharov-hatinoglu15}; however, 
for sets \( K \) subject to geometric constraints, 
it is often possible to place upper bounds on the Widom factors, 
and in some cases it is even possible to determine the precise asymptotic 
behavior of \(\calW_{n}(K)\) as \(n\rightarrow \infty\). 

The Faber polynomials, introduced in \cite{faber03}, have served as good trial polynomials to asymptotically saturate Szeg\H{o}'s 
inequality for sets with smooth boundaries. Throughout this text \(\Gamma\) denotes a Jordan curve 
and \(\Omega\) is the unbounded component of $\overline{\bC}\setminus \Gamma$. 
There is a canonical conformal map 
\[
\phi\colon \Omega \to \{ w \in\overline{\bC} : |w| > 1 \}
\] 
such that  \(\phi(\infty) = \infty\) and
\(
\phi(z)/z \rightarrow \capacity(\Gamma)^{-1}
\)
as \(z\rightarrow \infty\).
The \(n\)th Faber polynomial, denoted \(F_n\), is defined 
to be the polynomial part of the Laurent expansion of \(\phi(z)^n\) at 
\(\infty\), and so it  has leading term \(\capacity(\Gamma)^{-n}z^n\). 
In view of the extremal property of \(T_n\) we thus have
\begin{equation}
\calW_{n}(\Gamma) \leq \|F_n\|_\Gamma.
\label{eq:widom_faber_inequality}
\end{equation}
Faber proved in \cite{faber20} that if \(\Gamma\) is an analytic Jordan curve,
then the Faber polynomials are asymptotically minimal in the 
sense that \(\|F_n\|_\Gamma\rightarrow 1\) as \(n\rightarrow \infty\), and therefore
\begin{equation}
\lim_{n\rightarrow \infty}\calW_{n}(\Gamma) = 1.
\label{eq:asymptotic_saturation}
\end{equation}
Hence, if $\Gamma$ is an analytic Jordan curve, then the Faber polynomials 
are close to  being norm-minimal, and the Chebyshev polynomials asymptotically 
saturate Szeg\H{o}'s lower bound. Faber's result was subsequently
extended by Suetin to \(C^{1,\alpha}\)-smooth curves \cites{suetin64,suetin84}.

The question whether \eqref{eq:asymptotic_saturation} still holds for curves with corners has remained open. It appears that the main challenge is to find suitable trial polynomials when the Faber polynomials fail to be asymptotically minimal. Indeed, if \(\Gamma\) is assumed to be piecewise Dini-smooth with largest exterior angle \( \pi\lambda \), a combination of  \cite{pritsker99}*{Theorem~1.1} and \cite{pritsker02}*{Theorem 2.1} shows that
\begin{equation}
	\lambda\leq  \liminf_{n\rightarrow \infty}
	\|F_n\|_\Gamma\leq \limsup_{n\rightarrow \infty}\|F_n\|_\Gamma\leq 2.
	\label{eq:pritsker_bounds}
\end{equation}
Thus, at least when $\lambda>1$, we know that $ \liminf_{n\rightarrow \infty}
\|F_n\|_\Gamma>1$.
 
The second author and L. H\"{u}bner \cite{hubner-rubin25} recently carried out numerical experiments which 
suggest that the limit \eqref{eq:asymptotic_saturation} remains true in the presence of finitely many corners, including the case of cusps. Our main goal is to prove that this is indeed the case.

\begin{thm}
	If \(\Gamma\) is a piecewise Dini-smooth Jordan curve, then
	\[\lim_{n\rightarrow \infty}\calW_{n}(\Gamma) = 1.\]
	\label{thm:main}
\end{thm}
\vspace{-0.5cm}
In \cite{totik-varga15}*{Theorem 2.1}, Totik and Varga consider 
Chebyshev polynomials for a class of sets which includes piecewise smooth Jordan curves
as a special case. However,  their results focus on bounds for the Widom factors
and do not reveal the asymptotic behavior of \(\calW_{n}(\Gamma)\) as 
\(n\rightarrow \infty\), even when specialized to our setting. 	

We are also interested in obtaining a better understanding of the norm of  
Faber polynomials. Our next result provides a sharpening of the upper bound in \eqref{eq:pritsker_bounds}. We clarify that our usage of the word ``corner", say of exterior angle $\lambda\pi$, includes cusps ($\lambda=0,2$), as well as 
``flat" corners ($\lambda=1$). 
\begin{thm}
\label{thm:faber_upper_bnd}
Let \(\Gamma\) be a piecewise Dini-smooth curve with \(l\geq 1\) 
corners \(z_1,\ldots, z_l\) and associated exterior angles \(\lambda_1\pi,\dotsc,\lambda_l\pi \). With the quantities 
	\[
	\Lambda_k\coloneqq \max\{\lambda_k,2-\lambda_k\}, \qquad k=1,2,\ldots,l,
	\]
we have 
	\begin{align}\label{eq:faber_global_upper_bnd}
		\limsup_{n\to\infty} \|F_n\|_{\Gamma}\leq \max_{1\leq k\leq l}\Lambda_k.
	\end{align}
Additionally,
	\begin{align}\label{eq:pointwise_faber}
		\lim_{n\to\infty} \phi(z)^{-n}F_n(z)=\begin{cases}1,&\ z\in \Gamma\setminus \{z_1,\ldots,z_l\},\\
			\lambda_k,&\ z=z_k, \ k=1,\ldots,l,
		\end{cases}
	\end{align}
	the limit in \eqref{eq:pointwise_faber} holding uniformly on compact subsets of \(\Gamma\setminus \{z_1,\ldots,z_l\}\).
\end{thm}

If \(\Gamma\) bounds a convex domain, then it follows from \cite{pommerenke64}*{Satz 3} that 
\[\limsup_{n\rightarrow \infty}\|F_n\|_\Gamma = \max_{1\leq k \leq l}\lambda_k.\]
However, in the non-convex setting the only asymptotic upper bound we are aware of 
is given by \eqref{eq:pritsker_bounds}. Under the stronger Alper condition imposed on \(\Gamma\), it is shown in \cite{pritsker99} that \eqref{eq:pointwise_faber} holds in the pointwise sense. For piecewise analytic $\Gamma$, more precise asymptotics for $F_n$ away from the corners were obtained in \cite{minadiaz09}.

 Note that \eqref{eq:pointwise_faber} implies that there is equality in \eqref{eq:faber_global_upper_bnd} whenever there is $j\in\{1,\ldots,l\}$ such that 
\begin{align}\label{fabercond}
	\max_{1\leq k\leq l}\Lambda_k=\lambda_j.
\end{align}
When this happens,  we have $\limsup_{n\to\infty} \|F_n\|_{\Gamma}>1$. The exact value of this $\limsup$ when \eqref{fabercond} does not hold is unknown to us.
In Figure~\ref{fig:faber_inward} we illustrate the behavior of a Faber polynomial
associated with the circular lune \(C_{1/2}\), which has exterior angle \(\pi/2\).
In this case one has \(\Lambda_1 = \Lambda_2 = 3/2\) and \(\lambda_1 = \lambda_2 = 1/2\).
On the basis of this and similar plots for varying \(n\), we are led to believe that
\[
  1 <\limsup_{n\to\infty} \|F_n\|_{C_{1/2}} <3/2.
\]

\begin{figure}[h!]
	\centering
	\includegraphics[width=\textwidth]{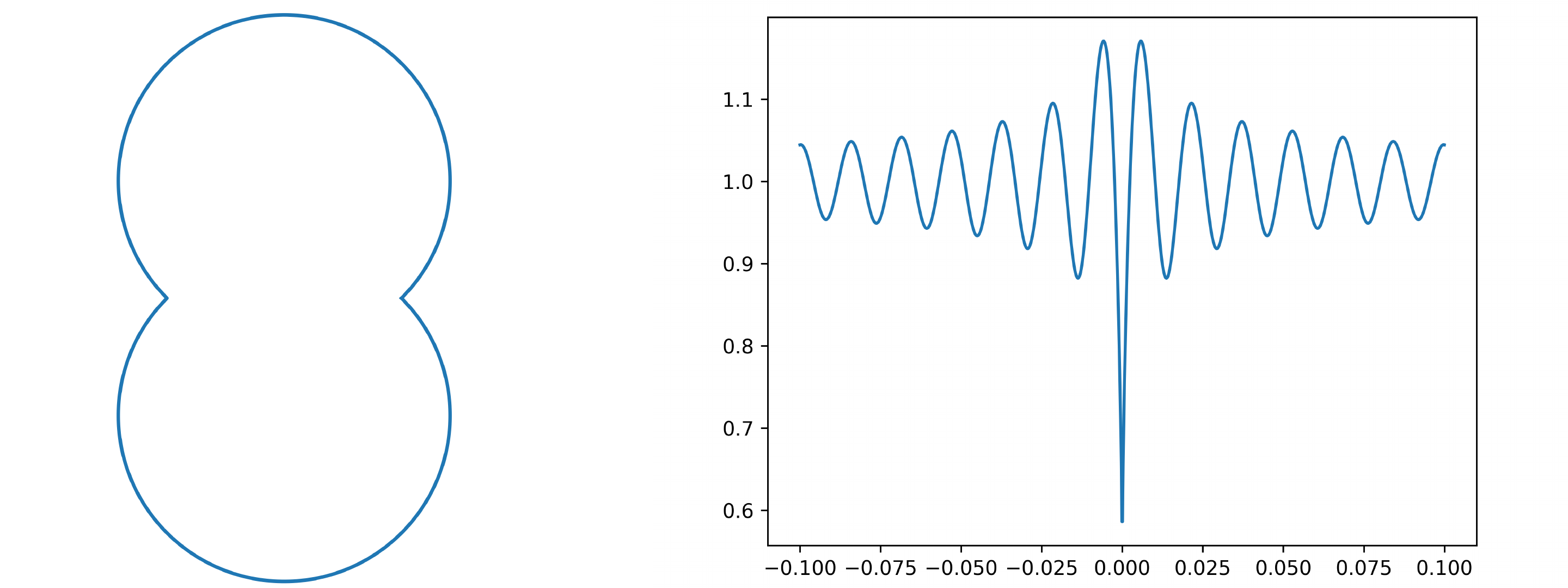}
	\caption{On the left -- the circular lune \(C_{1/2} = \{\psi(w): |w| = 1\}\) where \(\psi(w) = \frac{w+\sqrt{w^2-1}}{2}\). On the right -- the plot of the absolute value of the corresponding Faber polynomial \(t\mapsto |F_{401}(\psi(e^{2\pi it}))|\).}
	\label{fig:faber_inward}
\end{figure}
\begin{rem}
If $\Gamma$ is a Dini-smooth Jordan curve, a simplified version of the proof of Theorem \ref{thm:faber_upper_bnd} can be used to show that 
the uniform behavior in \eqref{eq:pointwise_faber} 
extends to the entire curve \(\Gamma\). Alternatively, this can also be deduced directly as a corollary of Theorem \ref{thm:faber_upper_bnd} since ``flat" corners are technically allowed. It follows that if \(\Gamma\) is a Dini-smooth Jordan curve, then 
	\[\lim_{n\rightarrow \infty}\|F_n\|_\Gamma = 1.\]
This extends \cite{suetin64}*{Theorem~2} and its refined version 
\cite{suetin84}*{Theorem 2, Section IV} for curves of class \(C^{1+\alpha}\), 
showing that the Faber polynomials are asymptotically minimal on Dini-smooth Jordan curves as well.  
\end{rem}

The asymptotic behavior in \eqref{eq:pointwise_faber} suggests how the Faber
polynomials can be modified to produce trial polynomials better suited for proving
Theorem \ref{thm:main}.
While the Faber polynomials have uniformly small norms on 
$\Gamma \setminus \{z_1,\dots,z_l\}$, they may exhibit larger absolute values near the corners.
This is visualized in Figure~\ref{fig:sidebyside} where we compare the Chebyshev and Faber polynomials for the 
deltoid curve $\Gamma \coloneqq \{ w + w^2/2 : |w| = 1 \}$.

\begin{figure}[hbtp]
\captionsetup{width=\linewidth}
\centering
\begin{subfigure}{0.49\textwidth}
\includegraphics[width=\linewidth]{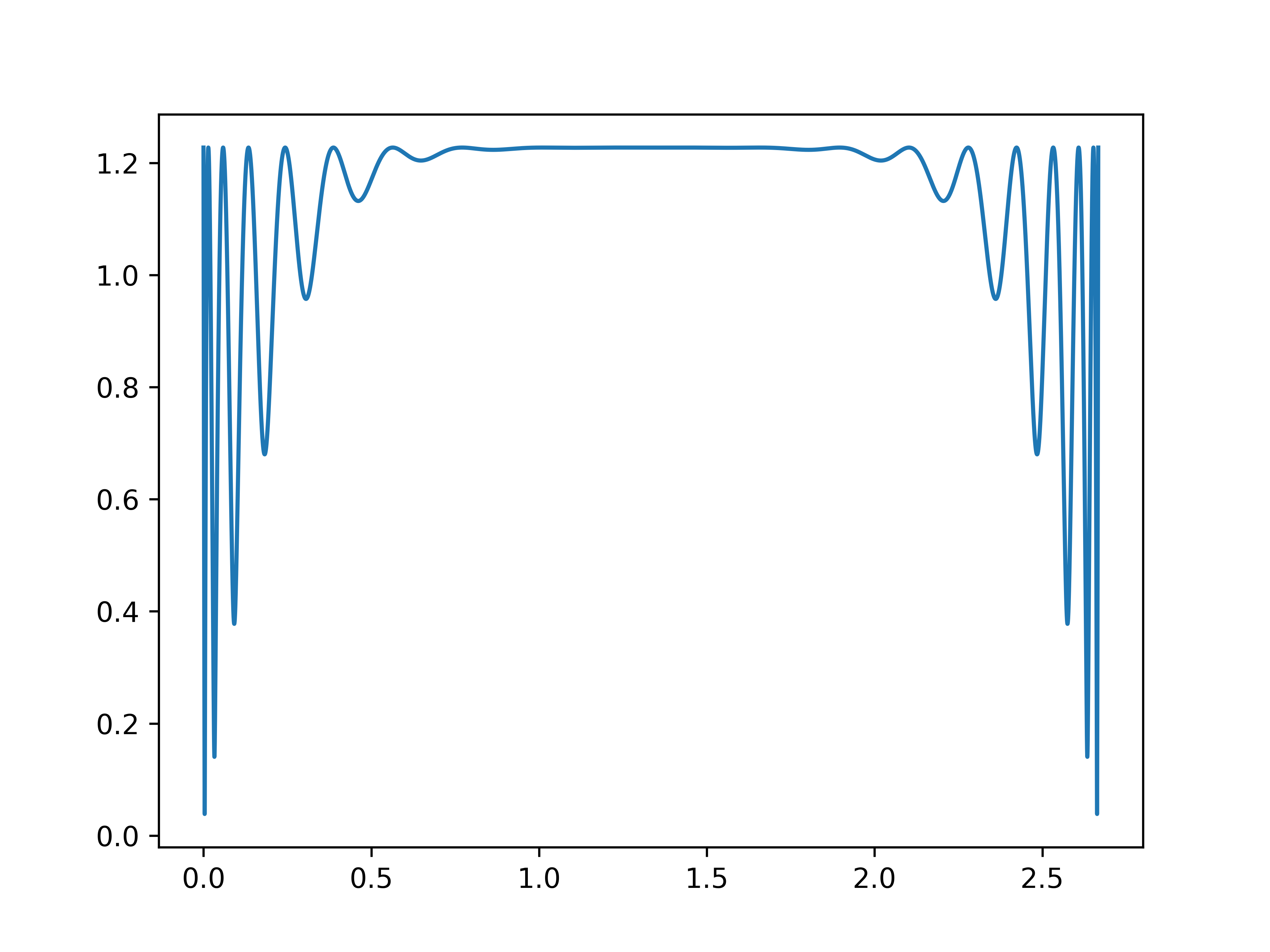}
\caption{\(t\mapsto |T_{30}(\gamma(t))|\)}
\label{fig:cheb}
\end{subfigure}
\hfill
\begin{subfigure}{0.49\textwidth}
\includegraphics[width=\linewidth]{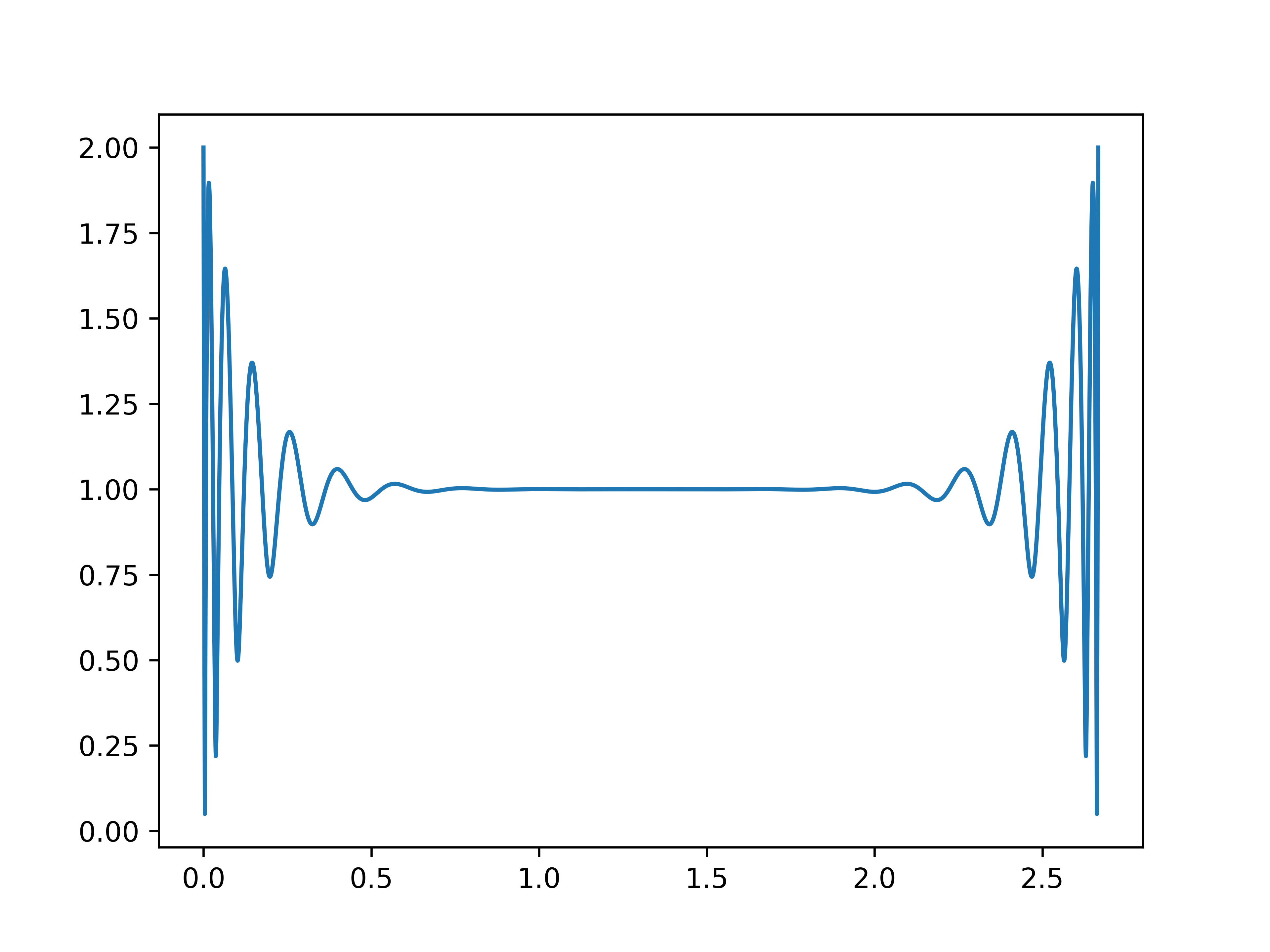}
\caption{\(t\mapsto |F_{30}(\gamma(t))|\)}
\label{fig:faber}
\end{subfigure}
  
\caption{Absolute values of the Chebyshev and Faber polynomials 
on the part of the deltoid between two cusps, parametrized by arc-length. 
The Chebyshev polynomials are computed using Tang's algorithm 
\cites{tang88,hubner-rubin25}.}
\label{fig:sidebyside}
\end{figure}

To construct polynomials with smaller norm, we look for a way to suppress the values of \(|F_n(z)|\) nearby the 
corner points---where $|F_n(z)|$ asymptotically deviates from 1---while allowing them to increase slightly in the mid-arc regions, where \(|F_n(z)| \to 1\) uniformly. This will be carried out by using weighted Faber polynomials \(Q_{n,m}\), defined as 
the polynomial part of expressions of the form \(G_m(z)\phi(z)^n\), where  \(G_m\)
is analytic in \(\Omega\) and \(G_m(\infty) = 1\). The polynomial \(Q_{n,m}\) is 
of degree \(n\) and leading coefficient \(\capacity(\Gamma)^{-n}\), and therefore
\begin{equation}
\calW_{n}(\Gamma)\leq \|Q_{n,m}\|_\Gamma.
\label{eq:faber_polynomials_widom_factors}
\end{equation}
We will see that it is possible to pick the weight functions \(G_m\) so as to have  
\begin{equation}\label{eq:faber_asymptotic_saturation}
\limsup_{n\rightarrow \infty}\|Q_{n,m}\|_\Gamma = 1+o(m),\qquad  m\to\infty.
\end{equation}
For instance, we could choose 
\[
G_m(z)=\prod_{k=1}^l\left(1-\frac{\phi(z_k)}{\phi(z)}\right)^{1/m},
\]
or more conveniently, as we do, a good approximation to this $G_m$ by a polynomial in $1/\phi(z)$.

\section{Constructing asymptotically minimal weighted Faber polynomials}
\label{sec:constructing_weighted}
\noindent In this section we construct a sequence of weighted Faber polynomials 
satisfying \eqref{eq:faber_asymptotic_saturation}, thus proving Theorem \ref{thm:main}. 
We first recall the appropriate regularity notions for the curve \( \Gamma \), and then 
state a useful Fourier analytic representation of the Faber polynomials initially 
proven by Pommerenke \cites{pommerenke64, pommerenke65}. This representation was 
further utilized by Gaier \cite{gaier99} and Pritsker \cites{pritsker02},
and we will use several ideas from these works. In order to make them applicable to our setting, 
we need to take extra care to ensure that certain error terms are uniform, which is only  implicit in the
above-mentioned works.

An arc with parametrization $\gamma(s)$, $s\in [0,1]$ is said to be Dini-smooth provided that $\gamma'(s)$ is continuous, never vanishes, and its modulus of continuity 
\[
\omega(t):=\max_{|s_1-s_2|\leq t}|\gamma'(s_1)-\gamma'(s_2)|,\qquad t\geq 0 
\]
satisfies the integrability condition 
\begin{align}\label{modulus-integrability}
	\int_0^1\frac{\omega(t)}{t}dt<\infty.
\end{align}
Note that \eqref{modulus-integrability} is weaker than assuming that
$\gamma'(s)$ satisfies a H\"{o}lder condition \[|\gamma'(s_1)-\gamma'(s_2)|\leq C|s_1-s_2|^\beta, \qquad 0<\beta\leq 1.
\]

A \emph{piecewise} Dini-smooth Jordan curve is a Jordan curve which can be decomposed into 
a union of finitely many Dini-smooth arcs.

\begin{figure}[t!]
	\captionsetup{width=.9\linewidth}
	\centering
	\includegraphics[width = 0.5\textwidth]{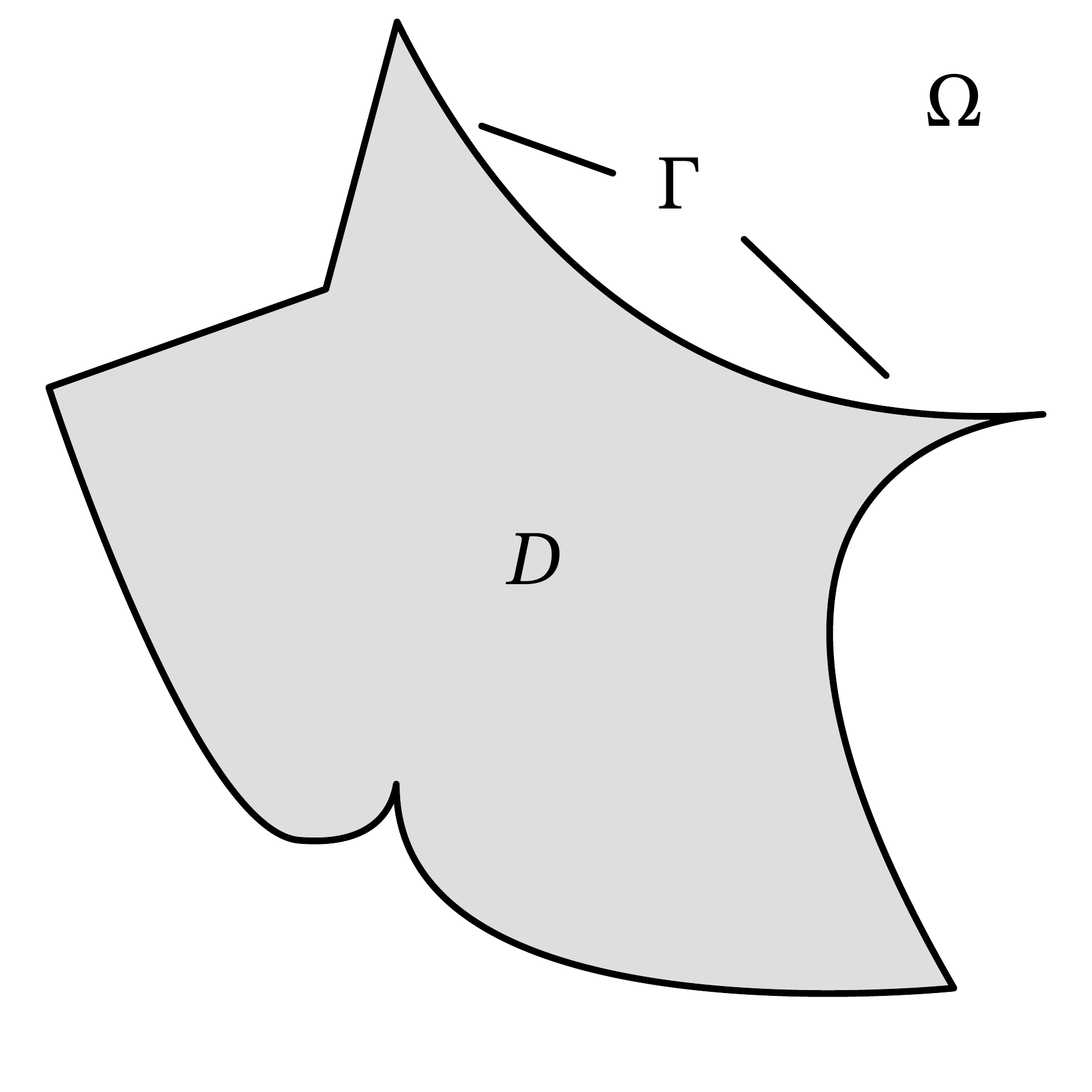}
%
%

%
%
	\caption{A piecewise smooth Jordan curve \(\Gamma\) with six corners, one of them an outer cusp, and one an inner cusp.}
	\label{fig:piecewise-v2}
\end{figure}
Let \(\Gamma\) denote a Jordan curve and let \(\Omega\) denote the 
component of \(\overline{\bC}\setminus \Gamma\) containing 
the point at infinity. The canonical exterior conformal map \(\phi:\Omega\rightarrow \{w\in \overline{\mathbb{C}}:|w|>1\}\) has an inverse,  denoted by $\psi(w)$, that extends to a homeomorphism of the exterior of the open unit disk onto $\overline{\Omega}$. The \(n\)th Faber polynomial \(F_n\) is the polynomial part of the Laurent expansion of \(\phi(z)^n\) at 
\(\infty\), and can be represented in integral form by 
\begin{align}\label{int-forms-Faber-Poly}
F_n(z)={} &\frac{1}{2\pi i}\int_{\Gamma_r}\frac{\phi(\zeta)^n}{\zeta-z}d\zeta=\frac{1}{2\pi i}\int_{|w|=r}\frac{w^n\psi'(w)}{\psi(w)-z}dw.
\end{align}
This representation is valid for \(r>1\) and \(z\) interior to the level curve \(\Gamma_r:=\{z\in\Omega:|\phi(z)|=r\}\).  
The second integral in \eqref{int-forms-Faber-Poly}   gives rise to the generating series for the Faber polynomials
\begin{align}\label{eq:faber_generating_series}
	\frac{w\psi'(w)}{\psi(w)-z}=1+\sum_{k=1}^\infty F_k(z)w^{-k},\qquad |w|>r> 1,
\end{align}	
valid for $z$ inside $\Gamma_r$, see e.g. \cite{suetin84}*{Chapter II}.

We will primarily make use of another integral representation for the Faber polynomials
due to Pommerenke (Theorem~\ref{thm:pommerenke_representation} below). 
Following \cites{pommerenke64,pommerenke65}, 
we fix \(\theta\in [0,2\pi)\), and define
\[
f_\theta(w):=\log\left(\frac{\psi(w)-\psi(e^{i\theta})}{aw}\right),\qquad |w|>1.
\]
This  \(f_\theta\) is analytic in \(|w|>1\), and for \(|w|>r>1\) we have the equality
\begin{align}\label{eq:first_pommerenke_form}
	f_\theta(w)=f_\theta(\infty)+\frac{1}{\pi }\int_{\alpha}^{\alpha+2\pi}
	\frac{ire^{it}w^{-1}}{1-re^{it}w^{-1}}\arg\left(\frac{\psi(re^{it})-\psi(e^{i\theta})}{are^{it}}\right)dt\,,
\end{align}
where \(\alpha\) is any real number such that \(\alpha<\theta<\alpha+2\pi\), 
see \cite{pommerenke65}*{Lemma~1}.

Let $\tau(w)$ be a conformal map of the open unit disk onto the interior domain $D$ of  the Jordan curve $\Gamma$, extended as a homeomorphism of the closed unit disk  onto $\overline{D}$. Let \(\ell\) be the straight line segment
between \(\tau^{-1}(\psi(e^{i\theta}))\) and \(\tau^{-1}(\psi(e^{i\alpha}))\) and consider the curve  
\[
C_\alpha:=\{\tau(w):w\in \ell 
\}\cup\{\psi(\rho e^{i\alpha}):\rho\geq 1\}.
\]
The open set \(U_\alpha:=\mathbb{C}\setminus C_\alpha\) is simply-connected 
and so there exists a branch of the argument \(\arg(z-\psi(e^{i\theta}))\) (as a function of \(z\)) 
on \(U_\alpha\). We take this branch of the argument for \(z\) on every level curve 
\(\{\psi(re^{it}):\alpha\leq  t<\alpha+2\pi\}\), \(r>1\), and on 
\(\Gamma\setminus \{\psi(e^{i\theta}),\psi(e^{i\alpha})\}\). At every point where any of these 
curves meet \(C_\alpha\) there is a \(2\pi\)-jump in the value of the argument. 

Accordingly, we can write \eqref{eq:first_pommerenke_form} as 
\begin{align}\label{eq:second_pommerenke_form}
	f_\theta(w)={} & f_\theta(\infty)+\frac{1}{\pi }\int_{\alpha}^{\alpha+2\pi}
	\frac{ire^{it}w^{-1}}{1-re^{it}w^{-1}}\left(\arg\left(\psi(re^{it})-\psi(e^{i\theta})\right)-t\right)dt,
\end{align}
whenever \(|w|>r>1\).

The curve \(\Gamma\) is said to be of \emph{bounded secant variation} (\(\Gamma\in BSV\)) 
provided that for every \(\theta\), the function \(t\mapsto \arg\left(\psi(e^{it})-\psi(e^{i\theta})\right)\)
is of bounded variation on \([\alpha,\theta)\cup (\theta,\alpha+2\pi]\).  We write
\[
v_\theta(t):=\arg\left(\psi(e^{it})-\psi(e^{i\theta})\right),\qquad t\in [\alpha,\theta)\cup (\theta,\alpha+2\pi].
\]
If \(\Gamma\in BSV\), then \(\Gamma\) has  one-sided tangents at each \(\psi(e^{i\theta})\). 
This implies that the function \(z\mapsto \arg(z-\psi(e^{i\theta}))\) remains 
bounded on \(\Omega\setminus C_\alpha\). We take the limit as \(r\to 1+\) 
in \eqref{eq:second_pommerenke_form} followed by integration by parts 
(see e.g. \cite{folland99}*{Theorem 3.36}) to get  
\begin{align}\label{thirdfaberintegral}
	f_\theta(w)={} & f_\theta(\infty)+\frac{1}{\pi }\int_{\alpha}^{\alpha+2\pi}
	\log(1-e^{it}w^{-1})dv_\theta(t)	,\qquad |w|>1,
\end{align}
where we also used the fact that 
\begin{align*}
	\int_{\alpha}^{\alpha+2\pi}\log(1-e^{it}w^{-1})dt=0	,\qquad |w|>1.
\end{align*}
Differentiating \eqref{thirdfaberintegral} with respect to \(w\) we find \begin{align*}
	\frac{w\psi'(w)}{\psi(w)-\psi(e^{i\theta})}
	=1+\sum_{k=1}^\infty\frac{w^{-n}}{\pi}\int_{\alpha}^{\alpha+2\pi}e^{int}dv_\theta(t),\qquad |w|>1,
\end{align*}
which in view of \eqref{eq:faber_generating_series} yields the following.
\begin{thm}[Pommerenke \cite{pommerenke64,pommerenke65}]
	\label{thm:pommerenke_representation}
	If \(\Gamma\in BSV\), then 
	\begin{align}\label{eq:faber_representation}
		F_n(\psi(e^{i\theta}))=\frac{1}{\pi}\int_{\alpha}^{\alpha+2\pi}e^{int}dv_\theta(t),\qquad n\geq 1.
	\end{align}
\end{thm}

Observe that \(v_\theta(t)\) is continuous on \([\alpha,\theta)\cup (\theta,\alpha+2\pi]\) and that 
\[
v_\theta(\theta+)-v_\theta(\theta-)=\lambda(\theta)\pi, 
\]
where \(\lambda(\theta)\pi\) is the exterior angle formed by the two 
one-sided tangents to \(\Gamma\) at \(\psi(e^{i\theta})\), 
 see e.g. \cite{pommerenke92}*{Section~3.5}. We set $v_\theta(\theta):=v_\theta(\theta+)$, so that    
\begin{equation}\label{eq:measure_decomposition}
dv_\theta=d\tilde{v}_\theta+\lambda(\theta)\pi \delta_\theta,
\end{equation}
where \(\delta_\theta\) is the unit point mass at the point \(\theta\) and $d\tilde{v}_\theta$ is the restriction of $dv_\theta$ to \([\alpha,\theta)\cup (\theta,\alpha+2\pi]\). If \(\Gamma\) 
is piecewise smooth, then $d\tilde{v}_\theta$ is absolutely continuous and is explicitly given by
\[
d\tilde{v}_\theta(t)= \operatorname{Re}
	\left(\frac{e^{it}\psi'(e^{it})}{\psi(e^{it})-\psi(e^{i\theta})}\right)dt.
	\]

We label the corners 
of \(\Gamma\) as \(z_1,\dotsc,z_l\), with corresponding points 
\(\theta_k\in [0,2\pi)\) such that \(\psi(e^{i\theta_k}) = z_k\), and define
\begin{equation}
	\Theta_\Gamma = \{\theta_k+2\pi j\, :\, k=1,\dotsc,l,\quad j \in \mathbb{Z}\}.
	\label{eq:def_theta_Gamma}
\end{equation}

Gaier \cite{gaier99}*{Theorem~4} proved that every piecewise Dini-smooth 
Jordan curve is of bounded secant variation. Building on the results of 
\cite{gaier99}, Pritsker bases his proof of the upper bound of 
\eqref{eq:pritsker_bounds} on the integral representation 
\eqref{eq:faber_representation} plus  a few other facts that are 
derived in the course of the proof of  \cite{pritsker02}*{Theorem 2.1}. 
These facts are collected in the following lemma. 
\begin{lem}
	\label{lem:pritsker_modification}
	Suppose \(\Gamma\) is a piecewise Dini-smooth Jordan curve.
	\begin{enumerate}[label=\textnormal{(\roman*)}]
		\item 	For every \(\delta>0\) sufficiently small, 
		\begin{align}\label{eq:pommerenke_riemann_lebesgue}
			\lim_{n\rightarrow \infty}\sup_{\theta\in [0,2\pi)}
			\left|\int_{\theta+\delta}^{\theta-\delta +2\pi}e^{int}dv_\theta(t)\right|=0.
		\end{align}
		\item	For all \(\epsilon>0\), there exists \(\delta>0\)  such that 
		\begin{align}\label{eq:pommerenke_upper_bnd_no_corner}
			(\theta-\delta,\theta+\delta)\cap \Theta_\Gamma=\emptyset\implies 
			\frac{1}{\pi}\int_{\theta-\delta}^{\theta+ \delta}|dv_\theta(t)|\leq 1+\epsilon,
		\end{align}
		\begin{align}\label{eq:pommerenke_upper_bnd_corner}
			(\theta-\delta,\theta+\delta)\cap \Theta_\Gamma\not=\emptyset 
			\implies \frac{1}{\pi}\int_{\theta-\delta}^{\theta+ \delta}|dv_\theta(t)|\leq 2+\epsilon. 	
		\end{align}
	\end{enumerate}
\end{lem}
In \eqref{eq:pommerenke_upper_bnd_corner}, $\delta$ is assumed small enough 
that for every $\theta$ the set  $(\theta-\delta,\theta+\delta)\cap \Theta_\Gamma$ 
contains at most one point. For the proof of the assertion (i) of the lemma, we refer to Pritsker \cite{pritsker02}*{p.~170-171}. The implications \eqref{eq:pommerenke_upper_bnd_no_corner} and \eqref{eq:pommerenke_upper_bnd_corner} are obtained during the proof of our Theorem \ref{thm:faber_upper_bnd}: \eqref{eq:pommerenke_upper_bnd_no_corner} follows from \eqref{pomm34} and \eqref{pomm36}, while \eqref{eq:pommerenke_upper_bnd_corner} is actually improved by the sharper bound \eqref{pomm41}.

We now proceed with the proof of our main theorem.
\begin{proof}[Proof of Theorem \ref{thm:main}]
	We will construct weighted Faber polynomails $Q_{n,m}(z)$, $n,m\in\mathbb{N}$, which for $r>1$ and $z$ interior to the level curve $\Gamma_r=\{z\in \Omega:|\phi(z)|=r\}$ can be expressed in integral form by (cf. \eqref{int-forms-Faber-Poly})
	\begin{align}\label{int-rep-for-Qnm}
	Q_{n,m}(z) =\frac{1}{2\pi i}\int_{\Gamma_r}\frac{P_{m}(\phi(\zeta))\phi(\zeta)^n}{\zeta-z}d\zeta=\frac{1}{2\pi i}\int_{|w|=r}\frac{P_{m}(w)w^n\psi'(w)}{\psi(w)-z}dw,
\end{align}
with $P_m(w)$	an analytic function in $|w|>1$ such that $P_m(\infty)=1$.  It is easy to see that this $Q_{n,m}$ is indeed the polynomial part of the Laurent expansion of \(P_m(\phi(z))\phi(z)^n\) at \(\infty\). Thus, $Q_{n,m}$ is of degree \(n\) and leading coefficient \(\capacity(\Gamma)^{-n}\), and consequently, we have the inequality
	\begin{equation}
	\calW_{n}(\Gamma)\leq \|Q_{n,m}\|_\Gamma,\qquad n\geq 1.
	\label{eq:widom_upper_bnd}
\end{equation}
We will show that it is possible to choose $P_m(w)$ in such a way that \eqref{eq:faber_asymptotic_saturation} holds. 
	
	Fix a positive integer \(m\) and consider, for \(0<r_m<1\), the function 
	\[g_m(w)\coloneqq\prod_{k=1}^{l}\left(1-\frac{r_m w_k}{w}\right)^{1/m},\qquad|w|>r_m,
	\]
	with the branches specified by \(g_m(\infty) = 1\). By choosing \(r_m\) 
	sufficiently close to \(1\) we can ensure the existence of \(\delta_m>0\) 
	such that 
	\begin{equation}
		|g_m(e^{i\theta})|<1/2, \qquad \theta\in (\theta_k-\delta_m,\theta_k+\delta_m),\quad k=1,\ldots, l.
		\label{eq:gm_close_to_corner}
	\end{equation}
	We take $\delta_m$ small enough that each segment \((\theta-\delta_m,\theta+\delta_m)\) 
	contains at most one point of \(\Theta_\Gamma\). 
	From the definition of \( g_m \) we readily see that
	\begin{equation}
		\max_{|w| = 1}|g_m(w)|\leq 2^{l/m}.
		\label{eq:gm_upperbound}
	\end{equation}
	Moreover, \(g_m\) has a Laurent expansion on $|w|>r_m$ of the form
	\[
	g_m(w) = 1+\sum_{j=1}^{\infty} a_jw^{-j},
	\]
	which converges uniformly on \(|w|\geq 1\). 
	Thus we can find a sufficiently large integer \(d_m\) such that
	\begin{equation}
		\max_{|w| = 1}\left|g_m(w) - \left(1+\sum_{j=1}^{d_m} a_jw^{-j}\right)\right|<1/m.
		\label{eq:laurent_diff}
	\end{equation}
	If we let
	\[P_m(w)\coloneqq 1+\sum_{j=1}^{d_m}a_jw^{-j},\]
	then it follows from \eqref{eq:gm_upperbound} and \eqref{eq:laurent_diff} that
	\begin{equation}
		\max_{|w| = 1}|P_m(w)|<m^{-1}+2^{l/m}.
		\label{eq:Pm_global_upper_bnds}
	\end{equation}
	Additionally,  from \eqref{eq:gm_close_to_corner} we get
	\begin{equation}
		|P_m(e^{i\theta})|<1/2+1/m, \qquad \theta\in (\theta_k-\delta_m,\theta_k+\delta_m),\quad k=1,\ldots, l.
		\label{eq:Pm_upper_bound_corner}
	\end{equation}
	
Let \(Q_{n,m}\) be the polynomial defined by \eqref{int-rep-for-Qnm}.	By the linearity of the integral, we obtain from  \eqref{int-rep-for-Qnm} and \eqref{int-forms-Faber-Poly} that
	\begin{align}
		\begin{split}
			Q_{n,m}(z) &= F_n(z)+\sum_{j=1}^{d_m}a_jF_{n-j}(z).
		\end{split}
		\label{eq:weighted_faber}
	\end{align}
	Using this together with Pommerenke's formula \eqref{eq:faber_representation} we obtain
	\begin{equation}
		Q_{n,m}(\psi(e^{i\theta})) = \frac{1}{\pi}\int_{0}^{2\pi}e^{int}P_m(e^{it})dv_\theta(t).
		\label{eq:weighted_faber_representation}
	\end{equation}
	
	Let \(\epsilon>0\) be given. By the second part of Lemma \ref{lem:pritsker_modification} 
	we can choose a \(\delta\in(0,\delta_m/2)\) so that
	\begin{equation}
		\frac{1}{\pi}\int_{\theta-\delta}^{\theta+\delta}|dv_\theta(t)|\leq \begin{cases}
			1+\epsilon, &(\theta-\delta,\theta+\delta)\cap\Theta_\Gamma = \emptyset \\
			2+\epsilon, &(\theta-\delta,\theta+\delta)\cap\Theta_\Gamma \neq \emptyset.
		\end{cases}
		\label{eq:variation_bds}
	\end{equation}
	We split the integral in 
	\eqref{eq:weighted_faber_representation} into two parts 
	\begin{equation}
		Q_{n,m}(\psi(e^{i\theta})) = \frac{1}{\pi}
		\left(\int_{\theta-\delta}^{\theta+\delta}+\int_{\theta+\delta}^{\theta+2\pi-\delta}\right)
		e^{int}P_{m}(e^{it})dv_\theta(t).
		\label{eq:split_integral}
	\end{equation}
	From \eqref{eq:Pm_global_upper_bnds}, \eqref{eq:Pm_upper_bound_corner} and \eqref{eq:variation_bds} we gather that
	\[\left|\frac{1}{\pi}\int_{\theta-\delta}^{\theta+\delta}e^{int}P_m(e^{it})dv_\theta(t)\right|\leq \begin{cases}
		\left(m^{-1}+2^{l/m}\right)(1+\epsilon), & (\theta-\delta,\theta+\delta)\cap\Theta_\Gamma=\emptyset, \\
		(2+\epsilon)(1/2+1/m),& (\theta-\delta,\theta+\delta)\cap\Theta_\Gamma\neq \emptyset,
	\end{cases}
	\]
	and conclude that
	\begin{equation}		\label{eq:local_upper_bnd}
		\sup_{\theta\in [0,2\pi)}\left|\frac{1}{\pi}
		\int_{\theta-\delta}^{\theta+\delta}e^{int}P_m(e^{it})dv_\theta(t)\right|
		\leq (1+\epsilon)\left(2^{l/m}+2/m\right).
	\end{equation}
		On the other hand, the first part of Lemma \ref{lem:pritsker_modification} implies that
	\begin{align}
		\begin{split}
			&\sup_{\theta\in [0,2\pi)}\left|\frac{1}{\pi}\int_{\theta+\delta}^{\theta-\delta+2\pi}
			e^{int}P_{m}(e^{it})dv_\theta(t)\right|\\& 
			= \sup_{\theta\in [0,2\pi)}\frac{1}{\pi}\left|
			\int_{\theta+\delta}^{\theta-\delta+2\pi}\left(e^{int}
			+\sum_{j=1}^{d_m}a_je^{i(n-j)t}\right)dv_\theta(t)\right|\rightarrow 0		
		\end{split}
		\label{eq:negligible_remainder}
	\end{align}
	as \(n\to \infty\). By combining \eqref{eq:local_upper_bnd} 
	and \eqref{eq:negligible_remainder} with \eqref{eq:split_integral} we obtain
	\[
	\limsup_{n\rightarrow \infty}\max_{z\in \Gamma}|Q_{n,m}(z)| 
	= \limsup_{n\rightarrow \infty}\max_{\theta\in [0,2\pi)}
	|Q_{n,m}(\psi(e^{i\theta}))|\leq (1+\epsilon)(2^{l/m}+2/m).
	\]
	Since \(\epsilon>0\) was arbitrary we can let it tend to \(0\) 
	and then \eqref{eq:widom_upper_bnd} gives us that
	\[\limsup_{n\rightarrow \infty}\calW_{n}(\Gamma)\leq 2^{l/m}+2/m.\]
	By letting \(m\to \infty\), and in view of Szeg\H{o}'s lower bound \eqref{eq:szego_inequality},  we obtain 
	\[\lim_{n\rightarrow \infty}\calW_{n}(\Gamma)= 1,\]
which concludes the proof.
\end{proof}

\section{Angular variation for Dini-smooth arcs}
\label{sec:variation_analysis}
\noindent We now study the local properties of the angular variation on a Dini-smooth arc. Our aim is to establish a refinement of 
\eqref{eq:pommerenke_upper_bnd_corner} necessary for the proof of our  Theorem~\ref{thm:faber_upper_bnd}. 

Any arc-length parametrization $z(s)$, $s\in [a,b]$ of a Dini-smooth arc (i.e., one satisfying $|z'(s)|=1$) is also Dini-smooth, and so is its extension to $(-\infty,\infty)$ given by
\begin{align}\label{extension-parametrization}
z(s):=\begin{cases}
	z(a)+(s-a)z'(a),&\ s<a,\\
	z(b)+(s-b)z'(b), & \ s>b.
\end{cases}
\end{align}
Note that this extension still satisfies the condition that $|z'(s)|=1$ 
and that the moduli of continuity of $z'(s)$ on the intervals $(-\infty,\infty)$ and $[a,b]$ are one and the same.

 The following result is due to Gaier \cite{gaier99}*{Theorem~5}.	
\begin{lem}
	Let \(z(s):[a,b]\rightarrow \bC\) be an arc-length parametrization of a 
	Dini-smooth arc. There exists \(\delta>0\) such that for every interval $I\subset [a,b]$ of length $|I|<\delta$ and every point \(s_0\in I\), we have 
	\begin{equation}
		\int_{I\setminus \{s_0\}}\left|d_s \arg\left(z(s)-z(s_0)\right)\right|
		\leq 8\int_{0}^{\delta}\frac{\omega(t)}{t}dt,
		\label{eq:variation_upper_bnd_mod_cont}
	\end{equation}
	where $\omega(t)$ is the modulus of continuity of $z'$.
	\label{lem:dini_arc_local_variation}
\end{lem}
The bounding constant \(8\) can be replaced by any value larger than \(4\), 
however, for our purpose \(8\) will suffice.

\begin{proof}We extend the parametrization $z(s)$ via \eqref{extension-parametrization} and prove the equivalent statement that there exists a $\delta>0$ such that for every $s_0\in [a,b]$, the inequality \eqref{eq:variation_upper_bnd_mod_cont} holds for $I=(s_0-\delta,s_0+\delta)$.
	
Because $z'$ is continuous, we can use the mean value theorem to prove the existence of a \(\delta>0\) such that 
	\begin{equation}\label{eq:diff_quot}
		\left|\frac{z(s)-z(s_0)}{s-s_0}-z'(s_0)\right|<\frac{1}{4}\quad \text{and} \quad \langle z'(s),z'(s_0)\rangle >0
	\end{equation}
	if \(|s-s_0|<\delta\) and \( s_0\in [a,b]\). Here \(\langle\cdot,\cdot\rangle\) 
	denotes the usual scalar product. Let us focus our attention on a single (but arbitrary) pair \((s_0,z(s_0))\).

	By applying a translation and a rotation, we may assume that \(z(s_0) = 0\) 
	and that $z'(s_0)>0$. In addition, by performing the change of variables $s\mapsto s+s_0$, we may also assume that $s_0=0$. After these transformations, the modulus of continuity of $z'$ and its unit modularity $|z'|=1$ remain unchanged. 
	
	Let us write $z(s)=x(s)+iy(s)$, so that $x'(0)=1$ and $y'(0)=0$. From \eqref{eq:diff_quot} we gather that if \(0<s<\delta\), then
	\begin{align}\label{IneqXandY}
		\left|\frac{x(s)}{s}-1\right|<\frac{1}{4}
	\end{align}
	and \(x'(s)>0\), which has the effect that \(\Re z(s)>0\) for \(0<s<\delta\). We can
	therefore define $\theta(s):= \arg z(s)=\Im(\log z(s))$, $0<s<\delta$, so that 
	\[
	|\theta'(s)|=\left|\Im \frac{z'(s)}{z(s)}\right|=\frac{\left|y'(s)x(s)-y(s)x'(s)\right|}{|z(s)|^2}.	
	\]
	Making use of the inequalities
	\begin{align}\label{pomm48}
	|x'(s)|\leq |z'(s)| = 1,\qquad |y(s)| \leq \int_{0}^{s}|y'(t)|dt
	\leq \int_{0}^{s}\omega(t)dt\leq s\omega(s),
	\end{align}
and \eqref{IneqXandY}, we find 
	\begin{align*}
		|\theta'(s)|&\leq\frac{ \left|y'(s)x(s)-y(s)x'(s)\right|}{x(s)^2}
		\leq\frac{\omega(s)}{x(s)}+\frac{s\omega(s)}{x(s)^2}\\
		& \leq \left(\frac{1}{(1-1/4)s}+\frac{1}{(1-1/4)^2s}\right)\omega(s)\leq 4\, \frac{\omega(s)}{s}.
	\end{align*}
	Therefore,
	\[
	\int_{(0,\delta)}\left|d_s\arg z(s)\right|
	= \int_{0}^{\delta}\left|\theta'(s)\right|ds\leq 4\int_{0}^{\delta} \frac{\omega(s)}{s}ds.
	\]
	By reversing the direction of the arc and considering \((-\delta,0)\) instead, the same argument yields 
	\[
	\int_{(-\delta,0)}\left|d_s\arg z(s)\right|\leq 4\int_{0}^{\delta} \frac{\omega(s)}{s}ds,
   	\]
	and since the bounds are based on \eqref{eq:diff_quot}, which holds uniformly in \(s_0\), the result follows.
\end{proof}

\begin{lem}
Let  \(z(s): [a,b]\rightarrow \bC\) be an arc-length parametrization of a 
	Dini-smooth Jordan arc \(\gamma\). For every \(\epsilon>0\), there exists \(\delta>0\) such that  
	\begin{equation}
		\int_I\left|d_s\arg\bigl(z(s)-\zeta\bigr)\right|\leq \pi+\epsilon 
		\label{eq:variation_estimate_outside}
	\end{equation}
	for every \(\zeta\in \bC \setminus \gamma\) and any interval \(I\subset [a,b]\) of length \(|I|<\delta\).
	\label{lem:variation_estimate_away_from_arc}
\end{lem}
\begin{proof}
	We again follow the ideas of Gaier, in particular the ones employed in 
	the proof of \cite{gaier99}*{Lemma~1}. One can verify that for any two 
	distinct points \(z_1,z_2\in \gamma\), the function \(\zeta\mapsto \arg(z_1-\zeta)-\arg(z_{2}-\zeta)\) 
	(which is well-defined as it is independent of the branch of \(z\mapsto \arg(z-\zeta)\) on \(\gamma\)) 
	is continuous on \(\bC\setminus \gamma\). By Lemma \ref{lem:dini_arc_local_variation}, 
	given \(\epsilon>0\) it is possible to find \(\delta>0\) such that 
	\[
	\int_{I\setminus \{s'\}}\left|d_s\arg\bigl(z(s)-z(s')\bigr)\right|<\epsilon
	\]
	for every interval \(I=[c,d]\subset [a,b]\) of length \(|I|=d-c<\delta\) 
	and point \(s'\in I\).  Consider a partition \(c=s_0<s_1<\ldots<s_N=d\) of \(I\). Since \(z'(s)\) is continuous, by decreasing \(\delta\) 
	if necessary we can guarantee that for every \(k=1,\ldots,N\) and \(s\in (s_{k-1},s_k)\), 
	the angle \(\vartheta^+_k(s)\) at \(z(s)\) of the triangle \(z(s_{k-1}),\, z(s),\, z(s_k)\) 
	and its conjugate  \(\vartheta^-_k(s)=2\pi -\vartheta^+_k(s)\) satisfy 
	\[
	\vartheta^\pm_k(s)<\pi +\epsilon.
	\]
	
	Let us introduce the function 
	\[
	h_N(\zeta)\coloneqq\sum_{k=1}^N\left|\arg\bigl(z(s_k)-\zeta\bigr)-\arg\bigl(z(s_{k-1})-\zeta\bigr)\right|.
	\]	
	For each \(k\), the function \(\arg(z(s_k)-\zeta)-\arg(z(s_{k-1})-\zeta)\) is a 
	branch of  \(\arg\left(\frac{z(s_k)-\zeta}{z(s_{k-1})-\zeta}\right)\), 
	thus it is harmonic on \(\overline{\bC}\setminus \gamma\). 
	It follows that \(h_N\) is subharmonic on \(\overline{\bC}\setminus \gamma\). 
	Because at each of its points the arc \(\gamma\) has a tangent line, 
	\(h_N\) is bounded on \(\overline{\bC}\setminus \gamma\). If \(s_{k_0-1}<s'<s_{k_0}\), then as \(\zeta\to z(s')\),
	\[
	h_N(\zeta)\to \sum_{k\not=k_0}|\arg(z_k-z(s'))-\arg(z_{k-1}-z(s'))|+\vartheta^\pm_k(s'),  
	\]
	the sign \(+\) or \(-\) chosen according to the side of \(\gamma\) the point \(\zeta\) approaches \(z(s')\) from. 
	It follows that 
	\[
	\limsup_{\zeta\to z(s')}h_{N}(\zeta)\leq 
	\int_{I\setminus\{s'\}}|d_s\arg(z(s)-z(s'))|+\pi+\epsilon<\pi+2\epsilon.
	\]
	By the maximum principle for subharmonic functions (see e.g. \cite{safftotik}*{Theorem 2.4}), we obtain 
	\(h_N(\zeta)\leq \pi+2\epsilon\) for all \(\zeta\in \mathbb{C}\setminus \gamma\), which in turn implies that 
	\[
	\int_{I\setminus \{s'\}}|d_s\arg(z(s)-\zeta)|\leq \pi+2\epsilon.\qedhere
	\]
\end{proof}
\begin{lem}	\label{lem:variation_upper_bnd}
	Assume that \(\zeta(t):[0,L_\zeta]\rightarrow \bC\) and \(z(s):[0,L_z]\rightarrow \bC\) 
	parametrize Dini-smooth Jordan arcs by arc-length that only intersect at \(z(0) = \zeta(0)\) 
	and that the smallest angle formed by the arcs is \(\mu\pi\in [0,\pi]\). 
	For every \(\epsilon>0\), there exists a \(\delta>0\) such that
	\begin{align}\label{pomm52}\int_{0}^{\delta}\left|d_s\arg\left(z(s)-\zeta(t)\right)\right|\leq \pi(1-\mu)+\epsilon,\quad 0<t<\delta.
		\end{align}
\end{lem}

\begin{figure*}[h!]
	\centering
	\includegraphics[width = 0.4\textwidth]{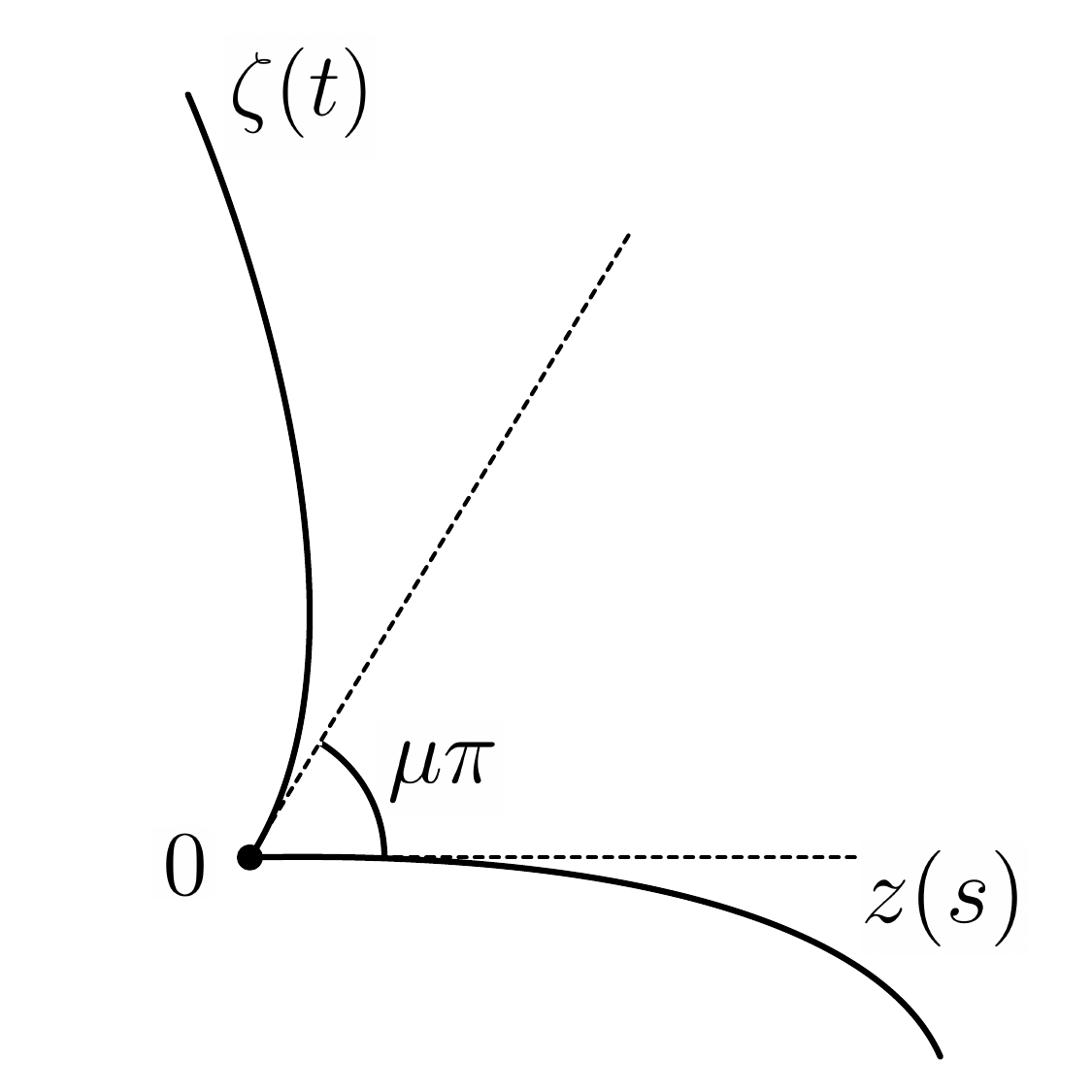}
\end{figure*}
\begin{rem}As  the proof of Lemma \ref{lem:variation_upper_bnd} will show, as far as the arc $\zeta$ is concerned, the condition that $\zeta(t)$ is Dini-smooth is unnecessary; all we need to assume is that $\zeta'(0)\not=0$.	
\end{rem}

\begin{proof}[Proof of Lemma \ref{lem:variation_upper_bnd}]
 After a translation and a rotation we may assume that $z(0)=0$ and that  $z'(0)=1$. Moreover, by applying a reflection about the $x$-axis, we may in addition assume that 
	\[
	\zeta(t) = te^{i\mu\pi}+o(t),\quad t\rightarrow 0+.
	\]

The case \(\mu = 0\) is already covered by Lemma \ref{lem:variation_estimate_away_from_arc}. We first suppose that \(0<\mu<1\). The idea of the proof is best conveyed through the example where the arc $z$ is the straight segment $z(s) = s$, $0\leq s\leq s_0$. In this case,
	\begin{align}\label{derivative-argument}
		\begin{split}
		\frac{d}{ds}\arg\left(z(s)-\zeta(t)\right) 
		& = \Im\left(\frac{d}{ds}\log\left(z(s)-\zeta(t)\right)\right) = \Im \frac{z'(s)}{z(s)-\zeta(t)}	\\
		& = \Im \frac{z'(s)\overline{\bigl(z(s)-\zeta(t)\bigr)}}{|z(s)-\zeta(t)|^2} 
		= \Im \frac{s-te^{-i\mu\pi}+o(t)}{|z(s)-\zeta(t)|^2}\\
		&=\frac{t\sin(\mu \pi)+o(t)}{|z(s)-\zeta(t)|^2}>0,\quad 0<t<\delta,
		\end{split}
	\end{align}
for some $\delta>0$ small. As a consequence, for every $\epsilon>0$ there is a corresponding $\delta$ such that if $0<t<\delta$, then
	\begin{align*}
		\int_{0}^{\delta}\left|d_s\arg\left(z(s)-\zeta(t)\right)\right| 
		& = \int_{0}^{\delta}\frac{d}{ds}\arg\left(z(s)-\zeta(t)\right)ds \\
		&= \underbrace{\arg\left(z(\delta)-\zeta(t)\right)}_{<0}
		-\underbrace{\arg\left(-\zeta(t)\right)}_{>\mu\pi-\pi-\epsilon}<\pi(1-\mu)+\epsilon.
	\end{align*}
	
	In the general case we cannot hope to prove that the derivative in \eqref{derivative-argument}	is positive for \(0<s,t<\delta\), where \(\delta\) is fixed and independent of \(t\). 
	However, we can obtain a local variant of monotonicity which is dependent of the point \(t\). Since
	\[\left|z(s) -s \right|\leq \int_{0}^{s}|z'(t)-1|dt\leq s\omega(s)\]
	we can write
	\begin{equation*}
		z(s) = s+E_1(s),\quad
		z'(s)  = 1+E_2(s),
	\end{equation*}
	where
	\[|E_1(s)|\leq s\omega(s),\quad |E_2(s)|\leq \omega(s)\]
	(the second inequality follows directly from the definition of \(\omega\) since \(z'(0)=1\)).
	Because \(\omega(s)\rightarrow 0\) as \(s\rightarrow 0\) and $z'(0)=x'(0)=1$, there is no restriction 
	in assuming that \(\omega(s)<1\) and \(x'(s)>0\).	We can further assume that \(\omega(s)>0\) for \(s>0\) since otherwise \(z'\) 
	is constant in some neighborhood of \(0\) and we are back in the example discussed at the beginning of the proof.
	
Arguing as in \eqref{derivative-argument}, we find
	\begin{align}
		\begin{split}
			\frac{d}{ds}\arg\left(z(s)-\zeta(t)\right) & = \Im \frac{z'(s)\overline{\bigl(z(s)-\zeta(t)\bigr)}}{|z(s)-\zeta(t)|^2}\\
			& = \Im\frac{\left(1+E_2(s)\right)\left(s+\overline{E_1(s)}-te^{-i\mu\pi}+o(t)\right)}{|z(s)-\zeta(t)|^2} \\
			& \geq \frac{t\sin(\mu\pi)+o(t)-\omega(s)(3s+t)}{|z(s)-\zeta(t)|^2}.
		\end{split}
		\label{eq:angle_derivative}
	\end{align}
	Observe that \(s\mapsto s\sqrt{\omega(s)}\) is strictly increasing from \(0\). 
	This implies that, for \(t>0\) small enough, there is a unique value \(\delta_t\) such that 
	\[\delta_t\sqrt{\omega(\delta_t)} = t.\]
	Since \(\omega(s)\rightarrow 0\) as \(s\rightarrow 0+\), we have 
	\[
	\frac{t}{\delta_t} = \sqrt{\omega(\delta_t)}\rightarrow 0,\quad \text{as}\quad t\rightarrow 0+.
	\]
	Using the inequality \eqref{eq:angle_derivative} we find that if 	\(0<s<\delta_t\) and \(0<t<\delta\) with \(\delta>0\) small enough, then  
	\begin{equation}
		\frac{d}{ds}\arg\left(z(s)-\zeta(t)\right) 
		\geq \frac{t\sin\mu +o(t)-t\omega(\delta_t)-3t\sqrt{\omega(\delta_t)}}{|z(s)-\zeta(t)|^2}>0.
		\label{eq:derivative_pos}
	\end{equation}
	But then \eqref{eq:derivative_pos} yields 
	\begin{align*}
		\int_{0}^{\delta}\left|d_s\arg\left(z(s)-\zeta(t)\right)\right| 
		& = \left(\int_{0}^{\delta_t}+\int_{\delta_t}^{\delta}\right)\left|\frac{d}{ds}\arg\left(z(s)-\zeta(t)\right)\right|ds \\
		& = \int_{0}^{\delta_t}\frac{d}{ds}\arg\left(z(s)-\zeta(t)\right)ds
		+\int_{\delta_t}^{\delta}\left|\frac{d}{ds}\arg\left(z(s)-\zeta(t)\right)\right|ds \\
		& = \arg\left(z(\delta_t)-\zeta(t)\right)-\arg\left(-\zeta(t)\right)
		\\
		&+\int_{\delta_t}^{\delta}\left|\frac{d}{ds}\arg\left(z(s)-\zeta(t)\right)\right|ds.
	\end{align*}
	Since 
	\begin{align*}
		z(\delta_t)-\zeta(t)={}&\delta_t+E_1(\delta_t)-t(e^{i\mu}+o(1))=\delta_t\left(1+O(\omega(\delta_t))+o(t)\right)
	\end{align*}
	as $t\to 0+$, we have that 
	\[
	\arg\left(z(\delta_t)-\zeta(t)\right)\rightarrow 0,\quad \text{as}\quad t\rightarrow 0+.
    \]	
	Since moreover \(\arg(-\zeta(t))\rightarrow \mu\pi   -\pi\) as \(t\rightarrow 0+\), we see that for a given $\epsilon >0$ there exists $\delta>0$ such that if \(0<t<\delta\), then  
	\begin{align}\label{pomm51}
	\int_{0}^{\delta}\left|d_s\arg\left(z(s)-\zeta(t)\right)\right|
	\leq \pi(1-\mu)+\epsilon/3+\int_{\delta_t}^{\delta}\left|\frac{d}{ds}\arg\left(z(s)-\zeta(t)\right)\right|ds.
	\end{align}
	
	What is left is to estimate the last integral. Writing 
	\[
	z(s) = x(s)+iy(s),\quad \zeta(t) = u(t)+iv(t),
	\]
	we obtain, using the triangle inequality, that
	\begin{align}
		\begin{split}
			\left|\frac{d}{ds}\arg\left(z(s)-\zeta(t)\right)\right| 
			&= \frac{\left|-x'(s)y(s)+x'(s)v(t)+y'(s)\left(x(s)-u(t)\right)\right|}{|z(s)-\zeta(t)|^2} \\
			& \leq \frac{|y'(s)|}{|z(s)-\zeta(t)|}+\frac{|y(s)|}{|z(s)-\zeta(t)|^2}+\frac{|x'(s)v(t)|}{(x(s)-u(t))^2},
		\end{split}
		\label{eq:angle_derivative_estimate}
	\end{align}
	where we have used that $|x'(s)|\leq |z'(s)|=1$. 
	Note that 
	\[ 
	\frac{|v(t)|}{x(\delta_t)-u(t)} = \frac{t\sin(\mu\pi)+o(t)}{\delta_t-t\cos(\mu\pi)+o(\delta_t)}\rightarrow
	0+,\qquad t\to 0+,
	\]
	and that if \(0<t<s\), then 
	\begin{align*}
	|z(s)-\zeta(t)| & = |s-te^{i\mu\pi}+o(s)| = |se^{-i\mu\pi}-t+o(s)|\\
	&\geq \left|\Im \left(se^{-i\mu\pi}-t+o(s)\right)\right|=s\sin(\mu\pi)+o(s).
	\end{align*} 
Thus by further restricting the size of \(\delta\) we can be sure that  
\[
0< \frac{|v(t)|}{x(\delta_t)-u(t)}<\frac{\epsilon}{3}, \qquad 0<t<\delta,
\]
\[
|z(s)-\zeta(t)|\geq\frac{s\sin(\mu\pi)}{2},\qquad 0<t<s<\delta,
\]
and that 
\[
\frac{6}{\sin^2 (\mu\pi)}\int_{0}^{\delta}\frac{\omega(s)}{s}ds<\frac{\epsilon}{3}.
\] 
Therefore, by \eqref{eq:angle_derivative_estimate}, \eqref{pomm48} and the fact that $t<\delta_t$, we have 
	\begin{align*}
		\int_{\delta_t}^{\delta}\left|\frac{d}{ds}\arg\left(z(s)-\zeta(t)\right)\right|ds 
		& \leq \int_{\delta_t}^{\delta}\frac{2\omega(s)}{s\sin (\mu\pi)}ds
		+\int_{\delta_t}^{\delta}\frac{4s\omega(s)}{s^2\sin^2(\mu\pi)}ds\\
		&	+|v(t)|\int_{\delta_t}^{\delta}\frac{x'(s)}{(x(s)-u(t))^2}ds.
 	\end{align*}
	The sum of the first two integrals is less than $\epsilon/3$. The final integral can be computed explicitly, giving us
	\begin{align*}
	|v(t)|	\int_{\delta_t}^{\delta}\frac{x'(s)}{(x(s)-u(t))^2}ds 
		&= \left.\frac{-|v(t)|}{x(s)-u(t)}\right\vert_{s = \delta_t}^{\delta}  = \frac{x(\delta)-x(\delta_t)}{x(\delta)-u(t)}\frac{|v(t)|}{x(\delta_t)-u(t)} \\
		&\leq \frac{|v(t)|}{x(\delta_t)-u(t)}<\epsilon/3.
	\end{align*}
After putting these estimates together with \eqref{pomm51} we obtain \eqref{pomm52}.

Suppose now that $\mu=1$. Then $\zeta(t)=u(t)+iv(t)=-t+o(t)$ as $t\to 0+$ and 
\begin{align*}
|z(s)-\zeta(t)|={} &(s+t)(1+o(1)),\qquad s+t\to 0+.
\end{align*}
It follows that for a given $\epsilon>0$, there exists $\delta>0$ such that 
\[
u(t)<0, \qquad \frac{|v(t)|}{-u(t)}< \frac{\epsilon}{2},\qquad 0<t<\delta,
\]
\[
 |z(s)-\zeta(t)|>\frac{s+t}{2},\qquad 0<s,t<\delta,
\]
and
\[
6\int_{0}^{\delta}\frac{\omega(s)}{s}ds< \frac{\epsilon}{2}.
\]
Using \eqref{eq:angle_derivative_estimate}, we see that for all $0<s,t<\delta$,
\begin{align*}
		\left|\frac{d}{ds}\arg\left(z(s)-\zeta(t)\right)\right|  \leq \frac{6\omega(s)}{s}+\frac{|v(t)|x'(s)}{(x(s)-u(t))^2}, 
	\end{align*}
	which after integration yields
 	\begin{align*}
 	\int_{0}^{\delta}\left|\frac{d}{ds}\arg\left(z(s)-\zeta(t)\right)\right|ds 
 	& \leq 6\int_{0}^{\delta}\frac{\omega(s)}{s}ds
 	+\frac{|v(t)|}{-u(t)}\frac{x(\delta)}{x(\delta)-u(t)}<\epsilon. \qedhere
 \end{align*}
\end{proof}

We now turn to proving Theorem \ref{thm:faber_upper_bnd}. We again let 
\(
z_1,\ldots,z_l
\)
denote the corners of the piecewise Dini-smooth curve \(\Gamma\) 
with \(\theta_k\in [0,2\pi)\), \(1\leq k\leq l\), such that
\[
z_k = \psi(e^{i\theta_k}),\qquad 1\leq k\leq l,
\]
and let \(\Theta_\Gamma\) be as in \eqref{eq:def_theta_Gamma}.

\begin{proof}[Proof of Theorem \ref{thm:faber_upper_bnd}]
		Let \(\lambda(\theta)\) be the exterior angle of \(\Gamma\) at \(\psi(e^{i\theta})\), so that
	\[
	\lambda(\theta)=\begin{cases}
		1,& \theta\not\in \Theta_\Gamma,\\
		\lambda_k,& \theta=\theta_k,\ 1\leq k\leq l.
	\end{cases}
	\]
	Observe that
	\begin{align}\label{estimate-for-lambda(theta)}
\lambda(\theta)	\leq \max_{1\leq k\leq l}\Lambda_k,
	\end{align}
	and that in view of \eqref{eq:measure_decomposition}, for all $\theta$ and $0<\delta<\pi$, we have 
	\begin{align}\label{pomm34}
		\frac{1}{\pi}\int_{\theta-\delta}^{\theta+\delta}|dv_\theta(t)|\leq \lambda(\theta)+	\frac{1}{\pi}\int_{(\theta-\delta,\theta+\delta)\setminus \{\theta\}}|dv_\theta(t)|.
	\end{align}
	By Theorem \ref{thm:pommerenke_representation}  (with \(\alpha=\theta-\delta\) and \(0<\delta<\pi\)) and \eqref{eq:measure_decomposition}, 
	we have 
	\begin{align}\label{pomm26}
		\left|F_n(\psi(e^{i\theta}))-\lambda(\theta) e^{in\theta}\right|
		\leq \frac{1}{\pi}\int_{(\theta-\delta,\theta+\delta)\setminus \{\theta\}}|dv_\theta(t)|
		+ \frac{1}{\pi}\left|\int_{\theta+\delta}^{\theta-\delta+2\pi} e^{int}\, dv_\theta(t)\right|.
	\end{align}
By Lemma \ref{lem:dini_arc_local_variation}, for every \(\epsilon>0\) there exists \(\delta_\epsilon>0\) such that for every interval  $(a,b)$ of length $b-a\leq 2\delta_\epsilon$ and $\theta \in (a,b)$, we have 
	\begin{equation}
		\left((a,b)\setminus \{\theta\}\right)\cap\Theta_\Gamma = \emptyset\implies 	\frac{1}{\pi}\int_{(a,b)\setminus \{\theta\}}|dv_\theta(t)|\leq \epsilon.
		\label{eq:no_intersection}
	\end{equation} 

Since for every \(\theta\) it is possible to find \(\delta\in(0,\delta_\epsilon)\) such that 
		\[
		\left((\theta-\delta,\theta+\delta)\setminus \{\theta\}\right)\cap\Theta_\Gamma = \emptyset,
	\]
it follows from \eqref{eq:no_intersection} that
	\begin{align}\label{pomm36}
	\frac{1}{\pi}\int_{(\theta-\delta,\theta+\delta)\setminus \{\theta\}}|dv_\theta(t)|\leq \epsilon.
\end{align}
This inequality together with \eqref{pomm26} 
	and \eqref{eq:pommerenke_riemann_lebesgue} lead to conclude that 
	\begin{align}\label{pomm37}
		\limsup_{n\to\infty}\left|F_n(\psi(e^{i\theta}))- \lambda(\theta) e^{in\theta}\right|\leq \epsilon,
	\end{align}
	proving \eqref{eq:pointwise_faber} in a pointwise sense. 
	Moreover, because for any compact set \(K\subset \Gamma\setminus \{z_1,\ldots,z_l\}\) 
	it is possible to find \(\delta>0\) such that  \(	(\theta-\delta,\theta+\delta)\cap \Theta_\Gamma=\emptyset\) whenever \( \psi(e^{i\theta})\in K\),	it follows that \eqref{pomm34} holds uniformly for all \(\theta\) with \(\psi(e^{i\theta})\in K\).

	The remaining task is to prove \eqref{eq:faber_global_upper_bnd}.  
	We pick \(\delta\in (0,\delta_\epsilon)\) so small that the distance between any two distinct elements of \(\Theta_\Gamma\) 
	is at least \(2\delta\). This has the effect that any interval \((\theta-\delta,\theta+\delta)\) 
	contains at most one point of \(\Theta_\Gamma\). 
	In view of \eqref{eq:no_intersection}, \eqref{pomm34} and \eqref{estimate-for-lambda(theta)} we have 	
\begin{align}\label{pomm35}
 \left((\theta-\delta,\theta+\delta)\setminus\{\theta\}\right)\cap \Theta_\Gamma= \emptyset\implies 	\frac{1}{\pi}\int_{\theta-\delta}^{\theta+\delta}|dv_\theta(t)|\leq \max_{1\leq k\leq l}\Lambda_k+\epsilon.
\end{align}
	Let us assume then that 
	\begin{align*}
 \left((\theta-\delta,\theta+\delta)\setminus\{\theta\}\right)\cap 	\Theta_\Gamma\neq \emptyset,
	\end{align*}
	say \(\theta_k\in (\theta,\theta+\delta)\) for definiteness. Note that in this situation we have $\lambda(\theta)=1$. By Lemma~\ref{lem:variation_upper_bnd}, 
	given \(\epsilon>0\) we have (after possibly shrinking \(\delta>0\)) the upper bound
	\begin{align}\label{pomm28}
		\begin{split}
			\frac{1}{\pi}\int_{\theta_k}^{\theta+\delta}|dv_\theta(t)| &= \frac{1}{\pi}\int_{\theta_k}^{\theta+\delta}\left|d_t\arg\left(\psi(e^{it})-\psi(e^{i\theta})\right)\right|\leq 1-\mu_k+\epsilon,
		\end{split}
	\end{align}
	where $\mu_k\pi $ is the smallest angle of $\Gamma$ at $z_k$. 
	By \eqref{eq:no_intersection} we also have 
	\begin{align}\label{pomm45}
		\frac{1}{\pi}\int_{(\theta-\delta,\theta_k)\setminus\{\theta\}}|dv_\theta(t)|\leq \epsilon.
	\end{align}
Adding \eqref{pomm28} and \eqref{pomm45} results in an inequality that when  coupled with \eqref{pomm34} yields   
	\begin{align}\label{pomm40}
		\frac{1}{\pi}\int_{\theta-\delta}^{\theta+\delta}|dv_\theta(t)|
		\leq 2-\mu_k+2\epsilon\leq \max_{1\leq k\leq l}\Lambda_k+2\epsilon
	\end{align}
	 (since \(\Lambda_k=2-\mu_k\)). 

Summarizing, we have found in \eqref{pomm35} and \eqref{pomm40} that for every value of $\theta$, it is true that 
	\begin{align}\label{pomm41}
	\frac{1}{\pi}\int_{\theta-\delta}^{\theta+\delta}|dv_\theta(t)|
\leq \max_{1\leq k\leq l}\Lambda_k+2\epsilon.
\end{align}
From \eqref{eq:faber_representation} we see that  
	\begin{align*}
		\left|F_n(\psi(e^{i\theta}))\right|
		&\leq \frac{1}{\pi}\int_{\theta-\delta}^{\theta+\delta}|dv_\theta(t)|
		+ \frac{1}{\pi}\left|\int_{\theta+\delta}^{\theta-\delta+2\pi} e^{int}\, dv_\theta(t)\right|,
	\end{align*}
 which in view of \eqref{pomm41} yields    
	\begin{align*}
		\|F_n\|_\Gamma
		= \sup_{\theta\in [0,2\pi)}\left|F_n(\psi(e^{i\theta}))\right|
		\leq \max_{1\leq k\leq l}\Lambda_k+2\epsilon	+ \sup_{\theta\in [0,2\pi)}\frac{1}{\pi}
		\left|\int_{\theta+\delta}^{\theta-\delta+2\pi} e^{int}\, dv_\theta(t)\right|.
	\end{align*}
We now use \eqref{eq:pommerenke_riemann_lebesgue} to conclude that  
	\[
	\limsup_{n\to\infty}\|F_n\|_\Gamma= \max_{1\leq k\leq l}\Lambda_k+2\epsilon,
	\]
	establishing \eqref{eq:faber_global_upper_bnd}.
\end{proof}

\bigskip

\noindent \textbf{Acknowledgment.} 
The research of Olof Rubin and Aron Wennman was supported by 
Odysseus Grant G0DDD23N from 
Research Foundation Flanders (FWO).

\bigskip
\bigskip

\noindent {Erwin Mi\~{n}a-D\'{\i}az}\\ 
Department of Mathematics, \\
The University of Mississippi, \\
Hume Hall 305, P. O. Box 1848, \\
MS 38677-1848, USA
\newline {\tt minadiaz@olemiss.edu}

\bigskip

\noindent {Olof Rubin} \\ Department of Mathematics, \\
KU Leuven, \\
Celestijnenlaan 200B, \\ 
3000 Leuven, Belgium
\newline {\tt olof.rubin@kuleuven.be}

\bigskip

\noindent {Aron Wennman} \\ Department of Mathematics, \\
KU Leuven, \\
Celestijnenlaan 200B, \\ 
3000 Leuven, Belgium
\newline {\tt aron.wennman@kuleuven.be}

\begin{bibdiv}
		\begin{biblist}
			
			\bib{achieser56}{book}{
				author={Achieser, N.~I.},
				title={Theory of approximation},
				publisher={Frederick Ungar publishing co. New York},
				date={1956},
			}
			
			\bib{andrievskii17}{article}{
				author={Andrievskii, V.~V.},
				title={On {C}hebyshev polynomials in the complex plane},
				date={2017},
				journal={Acta Math. Hungar.},
				volume={152},
				pages={505\ndash 524},
			}
			
			\bib{christiansen-simon-zinchenko-I}{article}{
				author={Christiansen, J.~S.},
				author={Simon, B.},
				author={Zinchenko, M.},
				title={Asymptotics of {C}hebyshev polynomials, {I}. subsets of {$\mathbb{R}$}},
				date={2017},
				journal={Invent. Math.},
				volume={208},
				pages={217\ndash 245},
			}
			
			\bib{christiansen-simon-zinchenko-III}{article}{
				author={Christiansen, J.~S.},
				author={Simon, B.},
				author={Zinchenko, M.},
				title={{Asymptotics of Chebyshev polynomials, III. Sets saturating Szeg{\H{o}}, Schiefermayr, and Totik--Widom bounds}},
				date={2020},
				journal={Oper. Theory Adv. Appl.},
				volume={276},
				pages={231\ndash 246},
			}
			
			\bib{christiansen-simon-zinchenko-IV}{article}{
				author={Christiansen, J.~S.},
				author={Simon, B.},
				author={Zinchenko, M.},
				title={{Asymptotics of Chebyshev polynomials, IV. Comments on the complex case}},
				date={2020},
				journal={J. Anal. Math.},
				volume={141},
				pages={207\ndash 223},
			}
			
			\bib{eichinger17}{article}{
				author={Eichinger, B.},
				title={Szeg{\H{o}}-{W}idom asymptotics of {C}hebyshev polynomials on circular arcs},
				date={2017},
				journal={J. Approx. Theory},
				volume={217},
				pages={15\ndash 25},
			}
			
			\bib{faber03}{article}{
				author={Faber, G.},
				title={{\"{U}ber polynomische Entwicklungen}},
				date={1903},
				journal={Math. Ann.},
				volume={57},
				pages={389\ndash 408},
			}
			
			\bib{faber20}{article}{
				author={Faber, G.},
				title={{\"{U}ber {T}schebyscheffsche Polynome}},
				date={1920},
				journal={J. Reine Angew. Math.},
				volume={150},
				pages={79\ndash 106},
			}
			
			\bib{fekete23}{article}{
				author={Fekete, M.},
				title={{\"{U}ber die Verteilung der Wurzeln bei gewissen algebraischen Gleichungen mit ganzzahligen Koeffizienten}},
				date={1923},
				journal={Math. Z.},
				volume={17},
				pages={228\ndash 249},
			}
			
			\bib{folland99}{book}{
				author={Folland, G.~B.},
				title={Real Analysis},
				edition={Second},
				series={Pure and Applied Mathematics (New York)},
				publisher={John Wiley \& Sons, Inc., New York},
				date={1999},
				note={Modern techniques and their applications, A Wiley-Interscience Publication},
			}
			
			\bib{gaier99}{article}{
				author={Gaier, D.},
				title={The {F}aber operator and its boundedness},
				date={1999},
				journal={J. Approx. Theory},
				volume={101},
				pages={265\ndash 277},
			}
			
			\bib{goncharov-hatinoglu15}{article}{
				author={Goncharov, A.},
				author={Hatino{\u{g}}lu, B.},
				title={Widom factors},
				date={2015},
				journal={Potential Anal.},
				volume={42},
				pages={671\ndash 680},
			}
			
			\bib{hubner-rubin25}{article}{
				author={H\"{u}bner, L.~A.},
				author={Rubin, O.},
				title={Computing {C}hebyshev polynomials using the complex {R}emez algorithm},
				date={2025},
				journal={Exp. Math.},
				pages={1\ndash 31},
			}
			
		
			\bib{minadiaz09}{article}{
			author={{Mi{\~{n}}a--D\'{i}az}, E.},
			title={On the asymptotic behavior of {F}aber polynomials for domains with piecewise analytic boundary},
			date={2009},
			journal={Constr. Approx.},
			volume={29},
			pages={421\ndash 448},
		}
			
		
			
			\bib{novello-schiefermayr-zinchenko21}{incollection}{
				author={Novello, G.},
				author={Schiefermayr, K.},
				author={Zinchenko, M.},
				title={Weighted {C}hebyshev polynomials on compact subsets of the complex plane},
				date={2021},
				booktitle={From operator theory to orthogonal polynomials, combinatorics, and number theory},
				publisher={Springer},
				pages={357\ndash 370},
			}
			
			\bib{pommerenke64}{article}{
				author={Pommerenke, Ch.},
				title={{\"U}ber die {F}aberschen {P}olynome schlichter {F}unktionen},
				date={1964},
				journal={Math. Z.},
				volume={85},
				pages={197\ndash 208},
			}
			
			\bib{pommerenke65}{article}{
				author={Pommerenke, Ch.},
				title={Konforme {A}bbildung und {F}ekete-{P}unkte},
				date={1965},
				journal={Math. Z.},
				volume={89},
				pages={422\ndash 438},
			}
			
			\bib{pommerenke92}{book}{
				author={Pommerenke, Ch.},
				title={Boundary Behaviour of Conformal Maps},
				publisher={Springer-Verlag Berlin Heidelberg},
				date={1992},
				series={Grundlehren der mathematischen Wissenschaften 299},
			}
			
				
			
			\bib{pritsker99}{article}{
				author={Pritsker, I.~E.},
				title={On the local asymptotics of {F}aber polynomials},
				date={1999},
				journal={Proc. Amer. Math. Soc.},
				volume={127},
				number={10},
				pages={2953\ndash 2960},
			}
			
			\bib{pritsker02}{article}{
				author={Pritsker, I.~E.},
				title={Derivatives of {F}aber polynomials and {M}arkov inequalities},
				date={2002},
				journal={J. Approx. Theory},
				volume={118},
				pages={163\ndash 174},
			}
			
				\bib{safftotik}{book}{
				author={Saff, E.~B.},
				author={Totik, V.},
			    series={A Series of Comprehensive Studies in Mathematics 316},
				title={Logarithmic Potentials with External Fields},
				publisher={Springer-Verlag Berlin Heidelberg New York},
				date={1997},
			}
			
			\bib{smirnov-lebedev68}{book}{
			author={Smirnov, V.~I.},
			author={Lebedev, N.~A.},
			title={Functions of a complex variable: constructive theory},
			publisher={The M.I.T. press, Massachusetts Institute of Technology, Cambridge, Massachusetts},
			date={1968},
		}
			
			\bib{suetin64}{article}{
				author={Suetin, P.~K.},
				title={Fundamental properties of {F}aber polynomials},
				date={1964},
				journal={Russ. Math. Surv.},
				volume={19},
				pages={121\ndash 149},
				note={Originally published in Russian: Uspehi Mat. Nauk {\bf{19}} (1964) no. 4 (118)},
			}
			
			\bib{suetin84}{book}{
				author={Suetin, P.~K.},
				title={Series of {F}aber polynomials},
				publisher={Gordon and Breach science publishers, Amsterdam, The Netherlands},
				date={1998},
			}
			
			\bib{szego24}{article}{
				author={Szeg\H{o}, G.},
				title={{Bemerkungen zu einer Arbeit von Herrn M. Fekete: \"{U}ber die Verteilung der Wurzeln bei gewissen algebraischen Gleichungen mit ganzzahligen Koeffizienten}},
				date={1924},
				journal={Math. Z.},
				volume={21},
				pages={203\ndash 208},
			}
			
			\bib{tang88}{article}{
				author={Tang, P.~T.~P.},
				title={A fast algorithm for linear complex {C}hebyshev approximation},
				date={1988},
				journal={Math. Comp.},
				volume={51},
				pages={721\ndash 739},
			}
			
			\bib{totik14}{article}{
				author={Totik, V.},
				title={Chebyshev polynomials on compact sets},
				date={2014},
				journal={Potential Anal.},
				volume={40},
				pages={511\ndash 524},
			}
			
			\bib{totik-varga15}{article}{
				author={Totik, V.},
				author={Varga, T.},
				title={Chebyshev and fast decreasing polynomials},
				date={2015},
				journal={Proc. Lond. Math. Soc.},
				volume={110},
				pages={1057\ndash 1098},
			}
			
			\bib{totik-yuditskii15}{article}{
				author={Totik, V.},
				author={Yuditskii, P.},
				title={On a conjecture of {W}idom},
				date={2015},
				journal={J. Approx. Theory},
				volume={190},
				pages={50\ndash 61},
			}
			
			\bib{widom69}{article}{
				author={Widom, H.},
				title={Extremal polynomials associated with a system of curves in the complex plane},
				date={1969},
				journal={Adv. Math.},
				volume={3},
				pages={127\ndash 232},
			}
			
		\end{biblist}
	\end{bibdiv}
\end{document}